\documentclass[12pt]{amsart}

\usepackage{amsmath, amssymb, amsthm, bm, mathtools, enumitem, caption}
\usepackage{hyperref}

\usepackage[noabbrev,capitalize]{cleveref}
\usepackage{adjustbox}
\crefname{equation}{}{}
\usepackage{ytableau}
\usepackage{graphicx, tikz}
\usepackage[textheight=8.7in, textwidth=6.5in]{geometry}

\usepackage{microtype}

\newtheorem{theorem}{Theorem}[section]
\newtheorem{lemma}[theorem]{Lemma}
\newtheorem{corollary}[theorem]{Corollary}

\newtheorem*{conjecture*}{Conjecture}

\theoremstyle{definition}

\theoremstyle{remark}
\newtheorem*{remark}{Remark}

\numberwithin{equation}{section}

\newcommand{\R}{\mathbb R}
\newcommand{\N}{\mathbb N}

\DeclareMathOperator{\Tr}{Tr}

\DeclareMathOperator{\Frob}{Frob_{\emph{p}}}

\DeclareMathOperator{\Gal}{Gal}

\DeclareMathOperator{\SL}{SL}
\DeclareMathOperator{\GL}{GL}
\DeclareMathOperator{\Img}{Im}

\newcommand{\E}{\mathcal E}

\newcommand{\C}{\mathbb C}

\newcommand{\Q}{\mathbb Q}

\newcommand{\Z}{\mathbb Z}

\title[Quasimodular Forms that detect primes are Eisenstein]{Quasimodular Forms that detect primes are Eisenstein}
\author{Jan-Willem van Ittersum, Lukas Mauth, Ken Ono, and Ajit Singh}

\dedicatory{In celebration of Hiroaki Nakamura's 60th birthday}

\address{University of Cologne, Department of Mathematics and Computer Science,
Weyertal 86-90, 50931 Cologne, Germany}
\curraddr{Korteweg--de Vries Institute for Mathematics, University of Amsterdam, Postbus 94248, 1090 GE  Amsterdam, The Netherlands}
\email{j.w.m.vanittersum@uva.nl}

\address{University of Cologne, Department of Mathematics and Computer Science,
Weyertal 86-90, 50931 Cologne, Germany}
\email{lmauth@uni-koeln.de}
\address{Dept. of Mathematics, University of Virginia, Charlottesville, VA 22904}
\email{ko5wk@virginia.edu}

\address{Dept. of Mathematics, University of Virginia, Charlottesville, VA 22904}
\curraddr{Dept. of Mathematics $\&$ Computing, Indian Institute of Technology (Indian School of Mines) Dhanbad, Jharkhand, India}
\email{ajit94@iitism.ac.in}

\begin{document}
\begin{abstract} 
MacMahon's partition functions and their extensions provide equations that identify prime numbers as solutions. These results depend on the theory of (mixed weight) quasimodular forms on $\SL_2(\Z)$. Two of the authors, along with Craig, conjectured an explicit description of the set of prime-detecting quasimodular forms in terms of Eisenstein series and their derivatives. Kane et al.\ recently verified this conjecture using analytic methods. We offer an alternative proof using the theory of $\ell$-adic Galois representations associated to modular forms.
 \end{abstract}

\maketitle

\section{Introduction and Statement of Results} 

 In a recent paper \cite{CVO2024}, two of the authors and Craig found expressions in partition functions that detect the set of prime numbers.
The simplest examples arise from  MacMahon's~\cite{MacMahon} functions
	\begin{align} \label{Ua} 
	\mathcal{U}_a(q)=\sum_{n\geq1}M_a(n)\,q^n:=\sum_{0< s_1<s_2<\cdots<s_a} \frac{q^{s_1+s_2+\cdots+s_a}}{(1-q^{s_1})^2(1-q^{s_2})^2\cdots(1-q^{s_a})^2}.
	\end{align}
	Clearly,
	$M_a(n)$ sums the  \emph{part multiplicity products} of the partitions of~$n$ with $a$ part sizes	\begin{equation}\label{M_def}
	M_a(n)=\sum_{\substack{0<s_1<s_2<\dots<s_a\\
			n=m_1 s_1+m_2s_2+\dots+m_a s_a}} m_1m_2\cdots m_a.
	\end{equation}
For $a\in \{1, 2, 3\},$ Theorem~1 of \cite{CVO2024} offers the two inequalities
\begin{displaymath}
\begin{split}
		&(n^2-3n+2)M_1(n)-8 M_2(n)\geq 0,\\
		&(3n^3 - 13n^2 + 18n - 8)M_1(n) + (12n^2 - 120n + 212)M_2(n) -960M_3(n) \geq 0.
\end{split}
\end{displaymath}
For $n\geq 2,$  these expressions vanish if and only if $n$ is prime.
	
There are infinitely many further prime-detecting expressions when one includes generalizations of MacMahon's $M_a(n),$ where for  a vector $\vec{a} = (v_1, v_2, \dots, v_a) \in \N^a$, we have the {\it MacMahonesque partition function}\footnote{These functions are also considered in Bachmann's master's thesis (see Chapter 5 of~\cite{BachmannMA}).}
\begin{equation}\label{M_general_def}
	M_{\vec{a}}(n):=\sum_{\substack{0<s_1<s_2<\dots<s_a\\m_1,\ldots,m_a>0\\
			n=m_1 s_1+m_2s_2+\dots+m_a s_a}} m_1^{v_1} m_2^{v_2}\cdots m_a^{v_a}.
\end{equation}
Theorem 6 of \cite{CVO2024} establishes the existence of infinitely many prime-detecting expressions in these functions.
Indeed, as $d\rightarrow +\infty,$ there are $\gg d^2$ many linearly independent inequalities of the form
$$
		\sum_{|\vec{a}|\leq d} c_{\vec{a}}\, M_{\vec{a}}(n)\geq 0,
$$
where $|\vec{a}|=v_1+v_2+\dots + v_a$ with $\vec{a}=(v_1, v_2, \dots, v_a)$,
that similarly detects primes. For example, we have the inequality
\begin{displaymath}
\begin{split}
		63M_{(2,2)}(n)-12&M_{(3,0)}(n) 
		-39M_{(3,1)}(n)-12M_{(1,3)}(n)\\
		&+80M_{(1,1,1)}(n)-12M_{(2,0,1)}(n)+12 M_{(2,1,0)}(n) +12 M_{(3,0,0)}(n)\geq 0.
\end{split}
\end{displaymath}

\begin{remark}
Further results on partition prime detection have recently been obtained by a number of authors, such as the works by Craig \cite{Craig}, Gomez \cite{Gomez}, and Kang et al.\ \cite{KMS}.
\end{remark}

All of these results are based on the phenomenon that the Fourier expansions of special (mixed weight) quasimodular forms detect primes. Here we address the problem of explicitly determining these forms. To make this precise, we begin with some necessary preliminaries.
For positive integers $k$ and $\tau\in \mathbb{C}$ with $\Im(\tau)>0,$ the weight $2k$ Eisenstein series (see Ch.~1 of~\cite{CBMS}) is
\begin{equation}\label{Eisenstein}
		G_{2k}(\tau) := -\frac{B_{2k}}{4k}+\sum_{n}\sigma_{2k-1}(n)q^n,
\end{equation}
	where $B_{2k}$ is the  $2k$-th Bernoulli number, $\sigma_{\nu}(n):=\sum_{d\mid n}d^{\nu},$ and $q:=e^{2\pi i \tau}.$ 
	It is well-known (for example, see \cite{Kaneko-Zagier}) that the quasimodular forms are given by polynomial ring  $\widetilde{M}:=\C[G_2,G_4,G_6].$ We also denote by $\widetilde{M}_k$ the subspace of $\widetilde{M}$ of forms with weight $k$, and $\widetilde{M}_{\leq k}$ the space of all quasimodular forms of mixed weight $\leq k$.

We require the differential operator
	\begin{align}
	D := \frac{1}{2\pi i} \dfrac{d}{d\tau} = q \dfrac{d}{dq}.
	\end{align}
	For $m \geq 0$, we have $D^m : q^n \mapsto n^m q^n$.
	Ramanujan proved the identities (for example, see \cite[Section 2.3]{CBMS})\footnote{These identities are often stated for the normalized Eisenstein series $E_k := 1 + \dots$.}
	\begin{align*}
	DG_2 = - 2G_2^2 + \dfrac{5}{6} G_4, \ \ \ D G_4 = -8 G_2 G_4 + \dfrac{7}{10} G_6, \ \ \ DG_6 = -12 G_2 G_6 + \dfrac{400}{7} G_4^2,
	\end{align*}
	which implies that $D$ is a derivation on $\widetilde{M}$, which increases weights by $2$ (i.e., $D:\widetilde{M}_{k}\to \widetilde{M}_{k+2}$).
	
	We determine the quasimodular forms whose Fourier expansions detect primes, a phenomenon that seems to have been first noticed by Leli\`evre \cite{Lelievre}. To this end, we let $\mathcal{E}$ denote the \emph{quasimodular Eisenstein space}, the subspace of quasimodular forms generated additively by the even weight Eisenstein series and their derivatives. We let $\Omega$ denote the {\it prime-detecting quasimodular forms}, where a quasimodular form
	$$
	f =\sum_{n\geq 0} b_n(f)q^n  \in \overline{\Q}[\![q]\!]
	$$
	is prime-detecting if $f\in \R[\![q]\!]$, for positive $n$ we have $b_n(f)\geq 0$, and for $n\geq 2$ vanish  if and only if $n$ is prime. Here we confirm Conjecture~2 of \cite{CVO2024}, which determines $\Omega$.
	
\begin{theorem}\label{Main-Theorem}
Every prime-detecting quasimodular form lies in the
quasimodular Eisenstein space $\mathcal{E}$. Equivalently, we have $\Omega \subset \mathcal{E}$.
\end{theorem}

\begin{remark}
Kane et al.\ recently proved Theorem~\ref{Main-Theorem} using analytic methods \cite{Kane-etal}. Here we offer an alternative proof using the theory of $\ell$-adic Galois representations.
\end{remark}

Thanks to Theorem~\ref{Main-Theorem} and \cite[Theorem~11]{CVO2024}, we obtain an explicit description of the prime-detecting quasimodular forms. For even $k\geq 6$, we define the weight $k$ distinguished quasimodular forms
\begin{align} \label{H definition}
	H_k := \sum_{n\geq 0} b_{n}(H_k) q^n :=
\begin{cases}
	\frac{1}{6}\left((D^2-D+1)G_2-G_4\right) &\ \ \ {\text {\rm if}}\ k=6, \\
	\frac{1}{24}(-D^2 G_{k-6} + (D^2+1) G_{k-4} - G_{k-2}) &\ \ \ {\text {\rm if}}\ k\geq 8.
\end{cases}
\end{align}
	
\begin{corollary}\label{Main-Corollary}
For all $n \geq 0$ and $k\geq 6$, we have $D^n H_k\in \Omega$. Conversely, if $f\in \Omega$, then $f$ is a linear combination of the forms  $D^n H_k$ for $n\geq 0$ and $k\geq 6$.
\end{corollary} 

To obtain these results, we modify earlier work of the  third author and Skinner \cite{Ono-Skinner}, itself an extension of the
classical work of Ribet
\cite{Ribet1, Ribet3, Ribet2} and Serre \cite{Serre}, which asserts that $\ell$-adic Galois representations associated to modular forms generally have ``large image." The main result we require is Lemma~\ref{Main-Lemma}, which we state and prove in Section~\ref{NutsBolts}. In Section~\ref{Proofs}, we employ it to prove Theorem~\ref{Main-Theorem} and Corollary~\ref{Main-Corollary}.

\section*{Acknowledgements}
We thank Koustav Banerjee and the referee for comments on a previous draft of this paper.
The first author thanks the support of the SFB/TRR 191 ``Symplectic Structure in Geometry, Algebra and Dynamics" grant funded by the Deutsche Forschungsgemeinschaft (Projektnummer 281071066 TRR 191).
The third author thanks the Thomas Jefferson Fund,  the NSF
(DMS-2002265 and DMS-2055118), and the Simons Foundation (SFI-MPS-TSM-00013279).  The  fourth author is grateful for the support of a Nehru-Fulbright postdoctoral fellowship.

\section{A fundamental lemma about mixed weight cusp forms}\label{NutsBolts}

This section is  devoted to the proof  of a fundamental lemma about a finite family of cuspidal Hecke eigenforms of level\footnote{With a slight modification of the hypotheses, the restriction to level 1 is not necessary for the conclusion.} $1$, say
$f_1,\dots, f_v$, whose Fourier expansions are
\begin{equation}\label{CuspForms}
f_i(\tau)=\sum_{n=1}^{\infty} a_{f_i}(n)q^n.
\end{equation}
These cuspidal Hecke eigenforms can have arbitrary weights. For sufficiently large primes~$\ell$, we show $\ell$-adically (see Lemma~\ref{Main-Lemma}) that we have maximal freedom for the congruence properties, simultaneously for the Hecke eigenvalues $a_{f_1}(p), \dots, a_{f_v}(p),$ as the primes $p$ vary in arbitrary arithmetic progressions of the form  $p\equiv r\pmod d,$ where $\gcd(r,d)=1.$

To demonstrate the usefulness of this general result, we consider the following toy situation:  Suppose that $f$ is a prime-detecting weight $2k$ modular form, which decomposes as the sum $f = G_{2k} + g,$ where  $g$ is a weight $2k$ cuspidal Hecke eigenform. Since it is prime-detecting, we have that the coefficients of $g$ with prime index must negate the divisor function arising from~$G_{2k}$. This is a contradiction, as we  then wouldn't have any freedom in the choice of the Hecke eigenvalues $\ell$-adically in arbitrary arithmetic progressions.

The proof of Theorem~\ref{Main-Theorem} is an extrapolation of this toy situation. Indeed,
thanks to a well-known decomposition theorem (for example, see (12) of \cite{Grabner}), every quasimodular form can be decomposed into a part which lies in the quasimodular Eisenstein space $\E$ (which includes derivatives of Eisesntein series) and a cuspidal part (including derivatives of cusp forms) of mixed weight. We will apply Lemma~\ref{Main-Lemma} to the cuspidal component.
If $\ell$ is sufficiently large, then a positive proportion of primes $p$ will realize any given
prescribed set of local congruence conditions $\ell$-adically for the Hecke eigenvalues.
 As a consequence, we will be able to show that the coefficients at certain primes cannot negate any divisor function (or linear combinations and derivatives thereof), which in turn implies Theorem~\ref{Main-Theorem}. This argument will be carried out in detail in Section~3.

This fundamental lemma is a modification of a result by the third author and Skinner on Fourier coefficients of arbitrary integer weight cusp forms on congruence subgroups \cite{Ono-Skinner}, which is based on the theory of Galois representaions attached to newforms developed by Eichler, Shimura, Serre, Deligne, and  Ribet (see for example \cite{Deligne, D-S, Ribet1,Ribet2, Serre}).

 To make this article self-contained, we restate the setup of \cite{Ono-Skinner} in the level {\bf $1$} case. Let  $\overline{\Q}$ be an algebraic closure of $\Q$, and for each rational prime $\ell$, let $\overline{\Q}_\ell$ denote an algebraic closure of the field of $\ell$-adic numbers $ \Q_\ell$. Fix an embedding $\overline{\Q}$ into $\overline{\Q}_\ell$. This determines a decomposition group $ D_\ell$ at $\ell$. In particular, if ${K}$ is a finite extension of ${\Q}$ with ring of integers $\mathcal{O}_{K}$, then for each $\ell$, this fixes a choice of a prime ideal $\mathfrak{p}_{\ell, {K}} \subset \mathcal{O}_{K} $ lying above $\ell$. Let ${F}_{\ell, {K}}$ be the residue field $\mathcal{O}_{K} / \mathfrak{p}_{\ell, {K}}$, and let $ | \cdot |_\ell $ be an extension to $\overline{\Q}_\ell $ of the usual $\ell$-adic absolute value on $\Q_\ell$.
	
Let $ f(\tau) = q+\sum_{n=2}^\infty a_f(n) q^n $ be a cuspidal normalized Hecke eigenform of level 1. Then each $ a_f(n)$ is an algebraic integer (see \cite[Theorem 2.29]{CBMS}), and thus the field generated by the coefficients $ \{a_f(n)\}_{n\geq 1}$ over $\Q $ is a finite extension, denoted by ${K}_f \subset \overline{\Q}$. 
If ${K}$ is any number field containing ${K}_f$, and if $\ell$ is any prime, then by the work of Eichler, Shimura, Deligne, and Serre (see for example \cite{Deligne, D-S,Serre}) there is a continuous, semisimple representation 
\begin{equation}
\rho_{f,\ell}:\Gal(\overline{\Q}/\Q)\rightarrow \GL_2(F_{\ell,{K}})
\end{equation}
for which
\begin{enumerate}[label=$(A_{\arabic*})$]
	\item $\rho_{f,\ell}$ is unramified at all primes $p\not= \ell$;
	
	\item We have that
	$$\Tr(\rho_{f,\ell}(\Frob))\equiv a_f(p) \pmod{ \mathfrak{p}_{\ell, {K}}}
	$$
	 for all primes $p\not= \ell$, where $\Frob$ denotes any Frobenius element for the prime $p$;
	
	\item  $\det \rho_{f,\ell}(\Frob)\equiv p^{k-1} \pmod{ \mathfrak{p}_{\ell, {K}}}$ for all primes $p\not= \ell$;
	
	\item $\det \rho_{f,\ell}(c)=-1$ for any complex conjugation $c$.
\end{enumerate}
 These representations capture essentially all of the properties of the reductions of the $a(n)$'s modulo $\mathfrak{p}_{\ell, K}$.
 For a positive integer $k_*$, we define
  $$\mathcal{N}_{k_*}:=\{f\in S_k(\SL_2(\Z)): f  \ \text{is a normalized Hecke eigenform with $k\leq k_*$}\}.$$ 
  For $f\in\mathcal{N}_{k_*}$, let $G_f\subset \Gal(\overline{\Q}/\Q)$ be the subgroup stabilizing ${K}_f$. For each $\ell$, let 
  $D_{f,\ell}:=G_f\cap D_\ell$ and ${F}_{f,\ell}:={F}_{\ell, {K}}^{D_{f,\ell}}$.
  We consider primes $\ell$ for which the following conditions hold.
  
 \begin{enumerate}[label=$(B_{\arabic*})$]
 	\item ${K}$ is unramified at $\ell$;
 	\item $\ell>2(k_*+2)$ and $\ell$ is odd;
 	\item For $f\in\mathcal{N}_{k_*}$, the representations $\rho_{f,\ell}$ are pairwise nonisomorphic;
 	\item The image of $\rho_{f,\ell}$ contains a normal subgroup $H_f$ conjugate to $\SL_2({F}_{f,\ell}),$ and suppose that $\Img(\rho_{f,\ell})/H_f$ is abelian.
 \end{enumerate}
 \smallskip
 
Condition $(B_4)$ tells us that the representation $\rho_{f,\ell}$ has large image, which is crucial for us to show that the eigenvalues can be freely chosen in the residue field. By the work of Serre~\cite{Serre} and Ribet~\cite{Ribet2},  property~$(B_4)$ holds for all sufficiently large primes $\ell$, as there are only finitely many forms in $\mathcal{N}_{k_*}.$

An important ingredient for the proof is that the cusp form in question is not a form with complex multiplication (a CM-form) in the sense of Ribet \cite{Ribet3}.\footnote{Ribet’s notion implies the more common definition of CM form where the $n$th Fourier coefficient is given as sums of values of Hecke characters over ideals of fixed norm $n.$} We consider cuspidal Hecke eigenforms of level $1$. Since the Hecke operators $T_n$ are self-adjoint, their eigenvalues are totally real and by multiplicativity all the Fourier coefficients of $f_i$ ($i=1,\ldots,v$) are real. For higher level $N > 1$ the space of cusp forms decomposes in a direct sum of spaces of cusp forms with Nebentypus $\varepsilon : (\Z / N\Z)^{\times} \rightarrow \C^{*}$. There is a simple criterion for a Hecke eigenform to have complex multiplication: the number field generated by the coefficients is a CM field unless $\varepsilon$ is trivial or of order two and satisfies $\varepsilon(p)a_{f_i}(p) = a_{f_i}(p)$ for all $p \nmid N$ prime (see for example \cite[Proposition 3.3]{Ribet3}). Moreover, we have that the absolute discriminant of the CM field divides $N$. As $N=1$, we do not need to consider CM forms. 

 For $v\in\mathbb{N}_0$ and $0 \leq i \leq v$, let $f_i(z) =q+ \sum_{n=2}^{\infty} a_{f_i}(n)q^n$ be the cuspidal Hecke eigenforms
 as in (\ref{CuspForms}), which are included in a set $\mathcal{N}_{k_*}.$
 Let $K_{f_i}\subset\overline{\Q}$ be the field containing all the Fourier coefficients of $f_i$ and $K$ be the compositum field of the $K_{f_i}$ for all $0 \leq i \leq v$. As before, let $\mathfrak{p}_{\ell, {K}} \subset \mathcal{O}_{K} $ be a prime ideal lying above $\ell$ and let ${F}_{\ell,K}$ denote the corresponding residue field.
 Assuming the notations above, we have the following lemma.
 
\begin{lemma}\label{Main-Lemma} If  a prime $\ell$  satisfies {\text {\rm ($B_1$)--($B_4$)}}, then
	for all pairs of integers $0<r<d$, with $\gcd(r,d)=1,$ and any set $w_1,\dots, w_v \in {F}_{\ell,K}$, a positive proportion of primes $p \equiv r \pmod{d}$ satisfies the congruence
	\[
	a_{f_i}(p) \equiv w_i \pmod{\mathfrak{p}_{\ell, K}}
	\]
	simultaneously for $0 \leq i \leq v$.
\end{lemma}
\begin{proof} Thanks to condition $(B_4)$, we have that the representation 
$$\rho_1 \times \cdots \times \rho_v$$
 contains a normal subgroup  conjugate to $v$ copies of $\SL_2({F}_{\ell,K})$, as is proven in \cite[Lemma~(i)]{Ono-Skinner} by modifying the arguments in \cite{Ribet1,Ribet3,Serre}. 
  Without loss of generality, we assume that
 $$H_1\times\cdots\times H_v=\SL_2({F}_{\ell,K})\times\cdots\times \SL_2({F}_{\ell,K}).
 $$
 Let $\varepsilon: \Gal\!\left(\overline{\mathbb{Q}}/\Q\right) \rightarrow (\Z/d\Z)^{\times}$ be the character defined by $\varepsilon(\sigma) = s\pmod d$, if $\zeta_d^\sigma = \zeta_d^s$. The condition $p \equiv r \pmod{d}$ is equivalent to $\varepsilon(\Frob) = r$. We claim that the image of 
 $$\rho_1\times \cdots \rho_v\times  \varepsilon
 $$
  contains $\mathrm{SL}_2({F}_{\ell,K})\times\cdots \times \mathrm{SL}_2({F}_{\ell,K})\times (\Z/d\Z)^{\times}$. 
  
 To this end, we denote the kernel of the map $\varepsilon$ by $G$ and write $\mathcal{H}:=G\cap H$, where 
 $$H:=\bigcap_{i=1}^v \rho_i^{-1}(H_i).
 $$
As shown in \cite[Lemma (i)]{Ono-Skinner}, we have that $\rho_i(H)=H_i$. We note that $\rho_i(\mathcal{H})\subseteq \rho_i(H)=H_i$. Moreover, observe that since $G$ is normal in $\Gal\!\left(\overline{\mathbb{Q}}/\Q\right)$, we have that $\mathcal{H}$ is normal in $H$. Hence, $\rho_i(\mathcal{H})$ is normal in the image of $H$ under $\rho_i$, which is $\rho_i(H)=H_i$. Now, as $H/\mathcal{H}$ and hence $H_i/\rho_i(\mathcal{H})$ are contained in the image of $\varepsilon$, these quotients are abelian.  As $H_i$ does not contain any non-trivial abelian quotients, we find that $\rho_i(\mathcal{H})=H_i$ as desired. 

Hence, let $g \in \Img(\rho_1 \times \cdots \times \rho_v)$ be conjugate to 
	\[
	\begin{pmatrix}
		0 & 1 \\
		-1 & w_1
	\end{pmatrix} \times \cdots \times \begin{pmatrix}
		0 & 1 \\ -1 & w_v
	\end{pmatrix}.
	\]
Then, $(g,1)$ admits a preimage in $\mathcal{H}$ under $\rho_1\times \cdots \times \rho_v\times \varepsilon$. Therefore, 
 by the Chebotarev Density Theorem, a positive proportion of the primes $p$ satisfy
 $$(\rho_1 \times \cdots \times \rho_v \times \varepsilon)(\Frob) = (g,1).
 $$
  For such primes $p,$ we have
	\[
	a_i(p) \equiv \Tr \begin{pmatrix} 0 & 1 \\ -1 & w_i \end{pmatrix} \equiv w_i \pmod{\mathfrak{p}_{\ell, K}} \quad \text{for}~1\leq i\leq v \quad\text{and}  \quad p \equiv \varepsilon(\Frob) \equiv 1 \pmod{d}. 
	\] 
	
For the more general case $p\equiv r \mod d$, let $\sigma \in \Gal\!\left(\overline{\mathbb{Q}}/\Q\right)$ be such that $\varepsilon(\sigma) = r$. Writing $\rho_i(\sigma) = \begin{psmallmatrix} a_i & b_i \\ c_i & d_i \end{psmallmatrix}$, we define $g=(g_1,\ldots,g_v) \in \Img(\rho_1 \times \cdots \times \rho_v)$ by
\[
g_i := 	\begin{cases} \begin{psmallmatrix}
		0 & 1 \\ {-1}  & \frac{w_i+b_i-c_i}{d_i}
	\end{psmallmatrix} & d_i \neq 0 \\[7pt]
	\begin{psmallmatrix}
	1 & 0 \\ \frac{w_i-a_i}{b_i}  & 1
	\end{psmallmatrix} & d_i= 0.
	\end{cases}
\]
Then, we have $\Tr(\rho_i(\sigma) \cdot g_i) = w_i$ and as before by the Chebotarev Density Theorem we obtain a positive density of primes $p$ with	
	\[
	a_i(p) \equiv\Tr(\rho_i(\sigma) \cdot g_i)\equiv w_i \pmod{\mathfrak{p}_{\ell, K}} \quad \text{for}~1\leq i\leq v \quad\text{and}  \quad p \equiv \varepsilon(\Frob) \equiv r \pmod{d}. 
	\] 

\end{proof}

\section{Proof of Theorem 1.1 and Corollary 1.2}\label{Proofs}

To prove Theorem~\ref{Main-Theorem}, we require the following two lemmas. The first lemma is about the Fourier coefficients of quasimodular forms for which the Fourier coefficients at primes vanish.  

\begin{lemma}\label{lem:rationalbasis}
 Let $f\in \widetilde{M}_{\leq k}$ be such that $a_p(f)=0$ for all primes $p$. Then, $f$ can be written as a $\C$-linear combination of $g_1,\ldots,g_d \in \widetilde{M}_{\leq k}\cap \Q[\![q]\!],$ where $a_p(g_j)=0$ for all  primes $p$.
\end{lemma}
\begin{proof}
	Let $f$ be as above. Since $\widetilde{M}_{\leq k}$ admits a basis $h_1,h_2,\ldots, h_m$ of forms with rational Fourier coefficients (monomials in the Eisenstein series $G_2, G_4$ and $G_6$), we write 
	\[
	f=\sum_{i}^m \alpha_i h_i \qquad \text{for some}~~ \alpha_i\in \C.
	\]
	
	Let $V_f$ be the vector space consisting of $\Q$-linear combinations of $\alpha_1,\alpha_2,\cdots, \alpha_m$. Writing $\beta_1,\ldots,\beta_d$ for a basis of this vector space, and $\alpha_i=\sum_{j} \alpha_{i,j} \beta_j$ with $\alpha_{i,j}\in \Q$, we obtain
	\[
	f=\sum_{i,j} \alpha_{i,j}\beta_j h_i = \sum_{j} \beta_j g_j, \qquad \text{where} \qquad g_j= \sum_{i} \alpha_{i,j} h_i.\]
	Note that the $g_j$'s are mixed weight quasimodular forms with rational Fourier coefficients. Now, since $f$ is prime-detecting, for all primes $p$ we have that
	\[
	\sum_{j} \beta_j\,a_{g_j}(p) = a_{f}(p) =  0.\]
	Since the $\beta_j$'s form a basis of $V_f$ and $a_{g_j}(p)\in \Q$, this implies that $a_{g_j}(p)=0$ for all primes $p$. Hence, we prove the lemma. 
\end{proof}
Our next lemma is standard and closely resembles the previous one, though it holds in greater generality. We include a proof for the sake of completeness.

\begin{lemma}\label{lem:Kcoef}
Let $F\subset\mathbb C$ be a subfield, and let
$h_1,\dots,h_n\in\widetilde M\cap F[\![q]\!]$ be linearly independent over
$\mathbb C$. Suppose that
\[
f=\sum_{i=1}^n \alpha_i h_i \in F[\![q]\!]
\]
for some $\alpha_i\in\mathbb C$. Then $\alpha_i\in F$ for all $i$.
\end{lemma}

\begin{proof} 
Let $V_f$ be a vector space consisting of $F$-linear combinations of the $\alpha_i$'s and all the elements of $F$. Writing $\beta_0=1, \beta_1,\ldots,\beta_d$ for a basis of this vector space, and $\alpha_i=\sum_{j} \alpha_{i,j} \beta_j$ with $\alpha_{i,j}\in F$, we obtain
$$ f=\sum_{i} \alpha_i h_i = \sum_{i,j} \alpha_{i,j} \beta_j h_i = g_0 + \sum_{j>0} \beta_j g_j,  \qquad \text{where} \qquad g_j= \sum_{i} \alpha_{i,j} h_i.$$
Note that the $g_j$ are mixed weight quasimodular forms with Fourier coefficients in $F$. Also, since the $\beta_j$'s form a basis for $V_f$ and $\beta_0=1$, we have that $\beta_j\not\in F$ for $j>0$. Hence, all the non-zero Fourier coefficients of $\sum_{j>0} \beta_j g_j$ do not lie in $F$.  However, since the coefficients of $f$ lie in $F$, we conclude that
$\sum_{j>0}\beta_j g_j = 0$, and so $f = g_0$. Therefore, we have
\[
\sum_{i=1}^n \alpha_i h_i = f = g_0 = \sum_{i=1}^n \alpha_{i,0} h_i,
\]
and hence
\[
\sum_{i=1}^n (\alpha_i - \alpha_{i,0}) h_i = 0.
\]
By the assumed $\C$-linear independence of $h_1,\dots,h_n$, we obtain
$\alpha_i - \alpha_{i,0} = 0$ for all $i$, so $\alpha_i = \alpha_{i,0} \in F$.
This completes the proof.
\end{proof}

Armed with lemmas, we are now in a position to prove Theorem~\ref{Main-Theorem}.

\begin{proof}[Proof of Theorem~\ref{Main-Theorem}] 	Suppose $f(\tau)=\sum_{n=0}^{\infty} a_f(n)q^n$ is a prime-detecting quasimodular form of mixed weight $\leq 2k$. By the decomposition theorem of quasimodular forms (for example, see (12) of  \cite{Grabner}), we have
	\begin{equation}\label{eq:qmfdecompositoin}
	\widetilde{M}_{2k} \,=\, \bigoplus_{r=0}^{k-1}\C D^{r}G_{2k-2r}\oplus \bigoplus_{r=0}^{k-1} D^rS_{2k-2r},
	\end{equation}
    where $S_{2k}$ is the space of cusp forms of level 1 and weight $2k$. 
   Since $f$ has algebraic Fourier coefficients, they lie in some number
field $K_0\subset\C$. Enlarging $K_0$ if necessary, we may assume that
$K_0$ contains the coefficient fields of all normalized cuspidal Hecke
eigenforms appearing in the decomposition (\ref{eq:qmfdecompositoin}). Therefore, $f$ and all the
forms $D^sG_m$ and $D^s f_{m,j}$ lie in $\widetilde M\cap K_0[[q]]$.
We now apply Lemma~\ref{lem:Kcoef} with $F=K_0$.

    By Lemma~\ref{lem:rationalbasis}, we can assume without loss of generality that $f$ has rational Fourier coefficients.
Therefore, we can write
	\[
	f = \sum_{s=0}^{k-1}  \sum_{m=2}^{2(k-s)} \alpha_{m,s} D^sG_{m} + \sum_{s=0}^{k-1}  \sum_{m=2}^{2(k-s)}  D^s g_{m,s},
	\]
	with coefficients $\alpha_{m,s}$ and where the $g_{m,s}$ are cuspidal modular forms of pure weight $m$ and level $1$. We write
	\[ g_{m,s} = \sum_{j} \beta_{m,s,j} f_{m,j}\]
	as sums of normalized cuspidal Hecke eigenforms of weight~$m$.
	Let $K$ be the compositum of the coefficient fields $K_{f_{m,j}}$.
All forms $D^sG_m$ and $D^s f_{m,j}$ occurring in the decomposition
of $f$ lie in $\widetilde M\cap K[[q]]$, and by the direct-sum
decomposition (\ref{eq:qmfdecompositoin}) they are linearly independent over $\mathbb C$.
 Lemma~\ref{lem:Kcoef} then implies that all coefficients $\alpha_{m,s}$ and $\beta_{m,s,j}$ are $K$-rational. As any multiple of a prime-detecting quasimodular form is prime-detecting, multiplying by a common denominator and dividing by the content, we may assume without loss of generality that the coefficients lie in~$\mathcal{O}_K$ and generate the unit ideal~$(1)$; that is, they are primitive in~$\mathcal{O}_K$. Now, by writing 
	\[h_{m,0}(x) = \sum_{s=0}^{k-1} \alpha_{m,s} x^s\in \mathcal{O}_K[x], \qquad h_{m,j} = \sum_{s=0}^{k-1} \beta_{m,s,j} x^s \in \mathcal{O}_K[x],\]
	we express our prime-detecting form $f$ as
	\[
	f=\sum_{m=2}^k h_{m,0}(D) G_m + \sum_{m=2}^k\sum_{j} h_{m,j}(D) f_{m,j}.
	\]

We now analyze this expression. The Eisenstein series has as coefficients divisor sums, so $a_{G_{m}}(p) = p^{m-1} + 1,$ 
 Therefore, the Fourier coefficients of $f$ with prime exponent $p$ are
\[
a_f(p) = h(p)+\sum_{m=2}^k \sum_{j} h_{m,j}(p) a_{f_{m,j}}(p) \qquad \text{with} \qquad h(x) = \sum_{m=2}^k h_{m,0}(x)(x^{m}+1).
\]
If all polynomials $h_{m,j}$ are the zero polynomial, the proof is complete since there is no cuspidal part. Otherwise, we fix a prime $\ell$ satisfying all of the properties $(B_1)-(B_4)$ and so that $\ell$ is larger than the degrees of all the polynomials $h_{m,j}$. Then, we let $0<r<\ell$ be such that there is a polynomial $h_{\mu,\iota}(x)$ which does not vanish at $x\equiv r \pmod \ell$. Since the polynomials admit less than $\ell$ different roots such an $r$ exists. 

 We apply Lemma~\ref{Main-Lemma} to the cusp forms $f_{m,j}$ with $d = \ell$. Therefore, we obtain a positive proportion of primes $p \equiv r \pmod \ell$
	that satisfy 
	\begin{align*}
	a_{f_{\mu,\iota}}(p) &\equiv h_{\mu,\iota}(r)^{-1}(1-h(r)) \pmod{\mathfrak{p}_{\ell, K}} \\
	a_{f_{m,j}}(p) &\equiv 0 \pmod{\mathfrak{p}_{\ell,K}} \qquad \text{ for all } (m,j)\neq (\mu,\iota).
	\end{align*}
	For those primes we thus find
	\[
	a_{f}(p) =h(p)+ h_{\mu,\iota}(p)  h_{\mu,\iota}(r)^{-1} (1-h(r)) = h(r)+ h_{\mu,\iota}(r)  h_{\mu,\iota}(r)^{-1} (1-h(r)) = 1  \pmod{\mathfrak{p}_{\ell,K}}. 
	\]
	Here, we used that the prime ideal $\mathfrak{p}_{\ell, K}$ of $\mathcal{O}_K$ lies above the rational prime $\ell$, i.e.\ $\mathfrak{p}_{\ell, K} \cap \Z = (\ell)$. In particular,  $h(p) \equiv h(r) \pmod{\mathfrak{p}_{\ell,K}}$. 
	Thus, $f$ does not detect primes. This is a contradiction, and so the cuspidal part of $f$ must vanish. 
Thus, every prime-detecting
quasimodular form $f\in\Omega$ has trivial cuspidal part, i.e.\ $f\in \mathcal{E}$.
This proves that $\Omega\subset \mathcal{E}$.
\end{proof}

\begin{proof}[Proof of Corollary~\ref{Main-Corollary}] Theorem~11 of \cite{CVO2024} characterizes $\mathcal{E}\cap \Omega$ in terms of the $D^nH_k,$ where $n\geq 0$ and $k\geq 6.$ Therefore, the corollary follows as we have that $\Omega=\mathcal{E}\cap \Omega.$
\end{proof}

\end{document}